\documentclass[11pt]{article}
\usepackage{amsmath, amssymb, amsthm, epsfig, graphics}
\newtheorem{theorem}{Theorem}[section]

\newtheorem{proposition}[theorem]{Proposition}
\newtheorem{remark}[theorem]{Remark}
\newtheorem{lemma}[theorem]{Lemma}
\newtheorem{corollary}[theorem]{Corollary}

\newtheorem{definition}[theorem]{Definition}

\title{Smallest eigenvalue distributions for two classes of $\beta$-Jacobi ensembles} 
\author{Ioana Dumitriu\footnote{Department of Mathematics, University of Washington, Seattle, WA. E-mail: dumitriu@math.washington.edu}}

\begin{document}

\maketitle

\abstract{We compute the exact and limiting smallest eigenvalue distributions for a class of $\beta$-Jacobi ensembles not covered by previous studies. In the general $\beta$ case, these distributions are given by multivariate hypergeometric ${}_2F_{1}^{2/\beta}$ functions, whose behavior can be analyzed asymptotically for special values of $\beta$ which include $\beta \in 2\mathbb{N}_{+}$ as well as for $\beta = 1$. Interest in these objects stems from their connections (in the $\beta = 1,2$ cases) to principal submatrices of Haar-distributed (orthogonal, unitary) matrices appearing in randomized, communication-optimal, fast, and stable algorithms for eigenvalue computations \cite{DDH07}, \cite{BDD10}.}

\section{Introduction}

This paper was born from the author's interest in the following problem.

\vspace{.5cm}

\noindent \textbf{Motivating Problem (MP)}. Given an $n \times n$ Haar-distributed orthogonal or unitary matrix $U_n$, choose a principal submatrix $M_r$ formed by $r$ rows and $r$ columns, $r \leq n/2$. Since the Haar (uniform) distribution is invariant under permutations of rows and columns, we might as well assume that $M_r$ is the upper-left corner principal submatrix (in Matlab notation, $M_r = U_n(1:r, 1:r)$). 
\begin{itemize}
\item[1.] What is the distribution of the smallest singular value of $M_r$? 
\item[2.] If $r, n \rightarrow \infty$, how does the singular value asymptotically depend on $n$ and $r$?
\item[3.] In particular, what kind of asymptotics do we get when $r = \Theta(n)$?
\end{itemize}

\vspace{.5cm}

The \textbf{MP} arose in the author's work in scientific computing, namely on (parallelizable, fast, and stable) randomized algorithms \cite{DDH07, BDD10}, since the quality of the randomized rank-revealing decomposition \textbf{RURV} presented there depends on the order of the said smallest singular value distribution. The main interest in that work is for the case when $n$ and $r$ are extremely large (the matrix $U_n$ itself does not fit into the processor's cache memory); hence Question 2. The reason why Question 3 is of particular importance is because, in that context, $r$ represents the split following a ``Divide-and-Conquer'' step in an eigenvalue algorithm (and we must expect that $r$ will be a fraction of the total number of eigenvalues $n$). 

The distribution, naturally, depends on whether the matrix is orthogonal or unitary. Question 1 relates to work by Collins \cite{collins03, collins05}. Particular instances of Question 2 have been solved by Jiang \cite{jiang09a, jiang09b}, who, when $r = o(n / \log n)$, has shown the much more general result that $M_r M_r^{T}$ converges entry-by-entry to a square Wishart distribution (for which the extremal eigenvalue asymptotics are known, see \cite{edelman89a}). 

We set out to examine Question 3, via Question 2; in the process, it transpired that Questions 2 and 3 both can be dealt with simultaneously. We have obtained thus exact answers for all $r$, $n$, as well as asymptotics for both $r$ fixed, $n$ growing to infinity, and for $r,n$ growing to infinity \emph{in any way}, including when $r = \Theta(n)$; then we noted that all results generalize in a way which we explain below to yield various asymptotics of extremal eigenvalues of special $\beta$-Jacobi ensembles. 

The $\beta$-Jacobi ensemble, defined below for general $\beta>0$, was the key to solving the \textbf{MP}.
\begin{definition} \label{defJ}
The numbers $\lambda_1, \lambda_2, \ldots, \lambda_m \in [0,1]$ are $\beta$-Jacobi distributed with parameters $a, b >-1$ if their joint distribution is given by 
\[
f_{\beta, a,b, m}(\lambda_1, \ldots, \lambda_m) =  \frac{1}{c_{\beta, a, b, m}} \prod_{i=1}^m \lambda_i^{\frac{\beta}{2}(a+1) - 1} ~(1- \lambda_i)^{\frac{\beta}{2}(b+1) - 1}~~\prod_{i <j} |\lambda_i - \lambda_j|^{\beta}~,
\]
with $c_{\beta, a, b, m}$ being the value of the corresponding Selberg integral, i.e.,
\begin{eqnarray}
c_{\beta, a, b, m} &= &\prod_{j=1}^{m} \frac{\Gamma\left (\frac{\beta}{2} ( a+j) \right) \Gamma \left ( \frac{\beta}{2} (b+j) \right) \Gamma \left ( 1+ \frac{\beta}{2} j\right)}{ \Gamma \left ( 1 + \frac{\beta}{2} \right) \Gamma \left ( \frac{\beta}{2} (a+b+m + j) \right)}~. \label{cee}
\end{eqnarray}
\end{definition}

When $\beta = 1,2$ (corresponding to the real, respectively complex cases), the scaled eigenvalues of $M_rM_r^{T}$ are $\beta$-Jacobi distributed (see Collins \cite{collins03, collins05},  Sutton \cite{sutton05}). 
This result is in fact more general, and we give here only the form in which we will use it in Section \ref{smed}.

\begin{proposition} \label{jac_distr} [following Theorem 5.1.3, \cite{sutton05}] 
Let $U_n$ be an $n \times n$ Haar-distributed orthogonal or unitary matrix, and let $M_r$ be the $r \times r$ upper left corner of $U_n$. Let $x_1, x_2, \ldots, x_r$ be the eigenvalues of $Y Y^{T}$. Then $\lambda_1:= x_1/\beta, \lambda_2:=x_2/\beta, \ldots, \lambda_r:=x_r /\beta$ follow the $\beta$-Jacobi distribution with parameters $a = 0$ and $b = n-2r$ ($\beta = 1$ in the real case, $\beta = 2$ in the complex one). 
\end{proposition}

Thus, the most important observation was that if one wants to find the answer to the \textbf{MP}, one should study the extremal eigenvalues of $\beta$-Jacobi ensembles. 

\subsection{$\beta$-Jacobi ensembles and extremal eigenvalue distributions: previous work}

The $\beta$-Jacobi ensembles (for $\beta = 1,2$) have been widely studied as the MANOVA real and complex distributions (see for example Muirhead \cite{muirhead82a} for the real case), and general $\beta$-analogues have been studied in the context of Selberg integrals \cite{selberg}, \cite{askey89a}, \cite{kaneko}, \cite{kadell_jacks} and log-gas theory (for a comprehensive study of the latter, see Forrester \cite{forr_book}). They have been given increasingly simple matrix models by Lippert \cite{ross}, Killip and Neciu \cite{killip-nenciu}, and most recently by Sutton \cite{sutton05} (the latter are the ones used here, in experiments). 

Many things are known about these distributions; exact results include the cumulative distribution functions (CDFs) for extremal eigenvalues \cite{kd08} for all $\beta$-Jacobi ensembles, as well as extremal eigenvalue asymptotics for special classes of $\beta$-Jacobi ensembles (see Jiang \cite{jiang09c}). In the latter paper, the author shows that the asymptotics for the largest eigenvalues are the same as for (special) $\beta$-Laguerre and $\beta$-Hermite ensembles, as heuristically shown by Edelman and Sutton \cite{sutton07} and then rigorously proved by Ram\'irez, Rider, and Vir\'{a}g \cite{rrv09} (such asymptotics are now known as the ``$\beta$-Tracy-Widom laws'', and they are ``soft-edge"). Finally, for $\beta = 1,2$, Johnstone \cite{johnstone08} has shown that, for a very large class of real or complex Jacobi ensembles, the extremal eigenvalues (at the ``hard edge") have regular Tracy-Widom \cite{tracy_widom_largest} asymptotics. For calculations of correlations functions (in general) and both kinds of edge-behavior (for $\beta \in 2\mathbb{N}$) and other useful information, see Forrester \cite{forr_book}).

In this paper, we examine the hard-edge $(a/r \rightarrow 0)$ behavior of certain Jacobi ensembles not considered before, and we find that (perhaps unsurprisingly) it agrees with the hard-edge behavior of certain corresponding Laguerre ensembles, some of whose limiting behavior was previously known, at least for certain values of $\beta$. The difference between the hard edge and the soft edge regime in this context is easy to explain; it essentially boils down to $(a/r,~b/r)$ converging to \emph{positive} constants (the latter), and one of them converging to $0$ (the former). 

For example, the Tracy-Widom asymptotics proved by Johnstone hold for real and complex Jacobi ensembles $(\beta = 1, 2)$ in the soft edge regime, whose extremal eigenvalues converge to numbers in $(0,1)$. In the cases of interest for us, as we will see, $a$ is actually constant, and thus $a/r \rightarrow 0$ (hard edge) and the smallest eigenvalues converge to $0$; unsurprisingly, the Tracy-Widom asymptotics are not valid here. In fact, we see the same phenomenon as for the related Laguerre ensembles: those ensembles for whom the extremal eigenvalues converge away from $0$ (soft edge) exhibit Tracy-Widom fluctuation asymptotics, but those for which the smallest eigenvalue converges to $0$ (hard edge) have either exponential ones (see, e.g., Edelman \cite{edelman89a}, $\beta = 1,2$), or given by a Bessel kernel (see Kuijlaars and Vanlessen \cite{kuijlaars_vanlessen02a}, for work on modified Jacobi ensembles when $\beta = 2$). 

\begin{remark} It is worth mentioning that there is a well-known simple procedure for ``turning'' $(a,b)$ $\beta$-Jacobi ensembles with eigenvalues $\lambda_1, \ldots, \lambda_n$ into $\beta$-Laguerre ones with parameter $a$ (the latter are $\beta$ generalizations of Wishart), by letting the parameter $b \rightarrow \infty$, and scaling the ensemble by letting $\lambda_i  \rightarrow x_i /b$ for all $i$.  See Forrester \cite[4.7.1]{forr_book}; the limiting procedure described there for the normalization constant also applies to the distributions themselves.

As a corrolary, this means that our asymptotical results also give both the exact and the asymptotical smallest eigenvalue distributions for the corresponding hard edge distributions in $\beta$-Laguerre ensembles, whose limits have been known to exist in terms of a stochastic Bessel operator (conjectured by Edelman and Sutton in \cite{sutton07} and proved by Ram\'irez and Rider in \cite{ramirez_rider09}). Many of these limits (for various $\beta$ and $a$) have been known; the cases we deal with here seem to be more general than anything else we were able to find in the literature (although for $\beta =1$ or an even integer, these cases are known \cite{forr_book}).

Although intuitively it makes sense that the $\beta$-Jacobi hard edge would always correspond to the $\beta$-Laguerre hard edge, one cannot move in the opposite direction, and claim that knowing the asymptotics for the latter will yield those for the former, since in that case the limiting procedure would have to be ordered: first $b \rightarrow \infty$, them $m \rightarrow \infty$. As such, our results, for which $m$ and $b+m$ tend to $\infty$ at any (dependent \emph{or} independent) rates, are more general. \end{remark}

The method employed here (use of generalized hypergeometric series) is the same method that we employed in \cite{kd08}; however, this time we apply it directly to  the probability density functions (PDFs) rather than the cumulative distribution functions (CDFs), and as a result the hypergeometric expressions we obtain are simpler, and allow us to do an asymptotical analysis, while the expressions obtained in \cite{kd08} do not yield easily to asymptotic study.

The rest of the paper is organized as follows. In Section \ref{hyp}, we define hypergeometric functions, Jack polynomials, and associated Pochhammer symbols. In Section \ref{smed}, we give a general hypergeometric expression for extremal eigenvalues distributions of Jacobi ensembles, and in Subsections \ref{case1} and \ref{case2} we calculate asymptotics in some special cases which correspond to generalizations of $\beta = 1$ and $\beta = 2$. In particular, this is where we answer the questions asked in the first paragraph. Finally, in Section \ref{numex} we illustrate  numerically how the distributions of Subsections \ref{case1}, \ref{case2} agree with Monte Carlo experiments. 

\section{Hypergeometric functions} \label{hyp}

Hypergeometric functions of multiple variables, also known as generalized hypergeometric functions, can be defined as formal power series, just like their classical counterparts. In the multiple variable case, the role of the monomials is taken by Jack polynomials indexed by partitions; if $X:= (x_1, \ldots, x_m)$, the Jack polynomial $C_{\kappa}^{\beta}(X)$ defined for a partition $\kappa$ and the variables $X$ is a symmetric, homogeneous polynomial. These polynomials satisfy various recurrences and systems of equations. For more information, see Chapter 12 of \cite{forr_book}, which is entirely dedicated to them. For computations of such quantities and connections to $\beta$-ensemble statistics, see the MOPs software package and associated paper \cite{dumitriu07b}, as well as Koev's hypergeometric function package and associated paper \cite{plamen_hyp}. In the cases $\beta = 1,2$, these polynomials are versions of the zonal polynomials (see \cite{muirhead82a}) and respectively scaled versions of the well-known Schur polynomials. 

To define Jack polynomials, we must first introduce two new quantities: the generalized Pochhammer symbol (or generalized rising factorial), and the symbol $j_{\kappa}^{\beta}$.
\begin{definition} \label{pocch_jk}
For any $\kappa = (k_1, \ldots, k_m)$ a partition of $k$ and for any $\beta>0$, 
\[
(a)_{\kappa}^{\beta}  \equiv \prod_{i = 1}^m \left ( a - \frac{\beta}{2}(i-1) \right)_{k_i}~,
\]
where $(x)_{k_i}$ is the classical rising factorial, $(x)_{k_i} = \Gamma(x+k_i)/\Gamma(x)$.  In addition, given the diagram of the partition $\kappa$ (see Figure \ref{parti}), we define for every square $s$ the ``arm-length'' $a_{\kappa}(s)$ as the number of squares to the right of $s$, and the ``leg-length'' of $s$ as the number of squares below $s$.  Then 
\[
j_{\kappa}^{\beta} \equiv  \prod_{s \in \kappa} \left (l_{\kappa}(s) + \frac{2}{\beta}(1+ a_{\kappa}(s) \right) \left ( l_{\kappa}(s) + 1 + \frac{2}{\beta} a_{\kappa}(s) \right )~.
\]
\end{definition}

\begin{figure}[ht]
\begin{center}
\includegraphics[height = 1.25in]{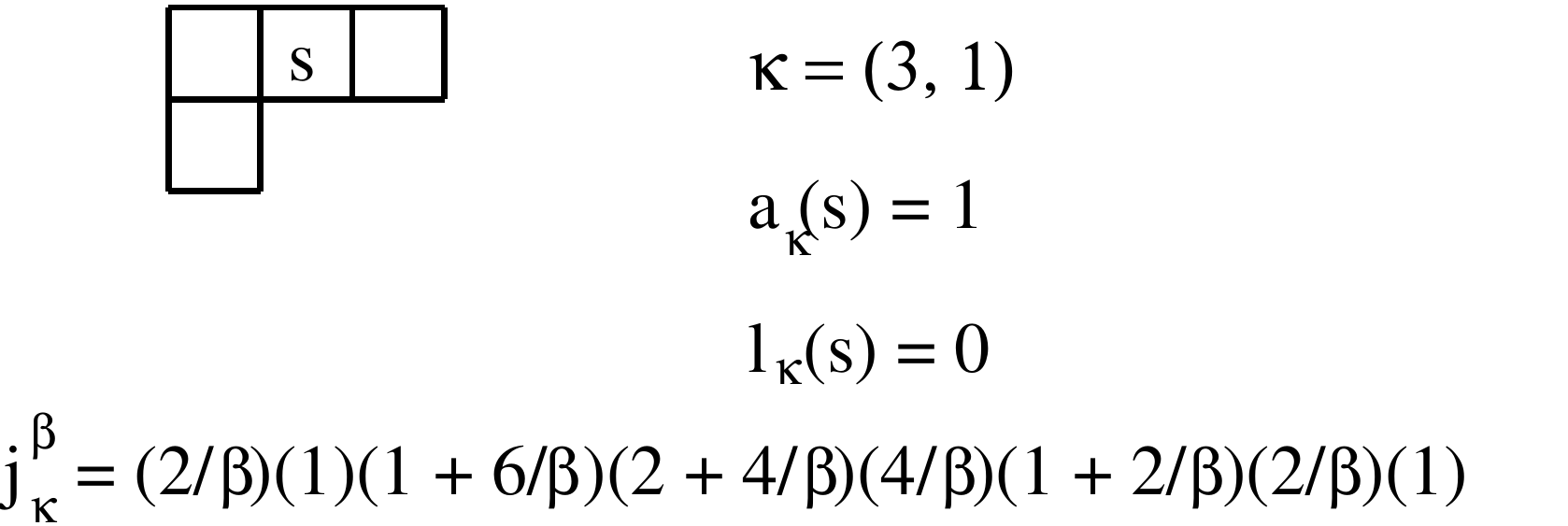}
\caption{Arm-length and leg-length for the square $s$, along with $j_{[3,1]}^{\beta}$.} \label{parti}
\end{center}
\end{figure}

The Jack polynomials have various normalizations; the one that we use here has the following two properties: for any integer $k$ and for $\kappa \vdash k$ a partition of $k$, and denoting by $I_m$ the vector of $m$ $1$s,
\begin{eqnarray}
\sum_{\kappa \vdash k} C_{\kappa}^{\beta} (X) & = & (x_1 + \ldots x_m)^k~, \nonumber \\
C_{\kappa}^{\beta}(I_m)  & = & \frac{\left ( \frac{2}{\beta} \right)^{2k} ~k!~ \left( \frac{m\beta}{2} \right)_{\kappa}}{j_{\kappa}^{\beta}}~. \label{jack_id}
\end{eqnarray}

Finally, we can now give the definition of the hypegeometric function of multiple variables, following \cite{forr_book}.

\begin{definition} The hypergeometric series ${}_pF_{q}$ is given by
\[
{}_pF_{q}^{\beta} (a_1, \ldots, a_p; b_1, \ldots, b_q; X) \equiv \sum_{k=0}^{\infty} \sum_{\kappa \vdash k} \frac{1}{k!} ~\frac{(a_1)_{\kappa}^{\beta} \ldots (a_p)_{\kappa}^{\beta}}{(b_1)_{\kappa}^{\beta} \ldots (b_q)_{\kappa}^{\beta}} ~ C_{\kappa}^{\beta}(X)~.
\]
\end{definition}

For the case  $p \leq q+1$, the series converges everywhere in the hypercube $\{|x_i|< 1, ~1 \leq i \leq n\}$. We will mostly focus on the cases $p = 2, q = 1$ (sometimes with $m=1$, in which case we recover the classical ${}_2F_{1}$ hypergeometric function of one variable), and sometimes $p = q= 1$ or $p = 0, ~q = 1$. It is worth mentioning two more facts about hypergeometric series, the first one which we quote from \cite{forr_book} (formula 13.5 with $p =2, q=1$), and the second one being an immediate consequence of the first one (given the homogeneity of Jack polynomials):
\begin{eqnarray}
\lim_{a \rightarrow \infty} ~{}_2F_{1}^{\beta}(a, b; c; X/a) &=& {}_1F_{1}^{\beta}(b;c;X)~, ~~\mbox{and} \label{fa1}\\
\lim_{a, b \rightarrow \infty} ~{}_2F_{1}^{\beta}(a,b;c;X/(ab)) & = & {}_0F_{1}^{\beta}(c; X)~. \label{fa2}
\end{eqnarray}

\section{Smallest eigenvalue distributions} \label{smed}

\subsection{Most general case} \label{mgc}

We are interested in finding an expression for the probability density function for the smallest (respectively, largest) eigenvalues of the Jacobi ensemble of parameters $a$ and $b$ (see Definition \ref{defJ}).


We would like to obtain the marginal distribution $f_{min}$ for the smallest eigenvalue; assume $\lambda = \lambda_m$. To that extent, we will integrate out all other variables but $\lambda_m$, to obtain 
\begin{eqnarray*}
f_{min}(\lambda)  &=& \frac{m}{c_{\beta, a,b, m}} ~\lambda^{\frac{\beta}{2}(a+1)-1} ( 1- \lambda)^{\frac{\beta}{2}(b+1) -1} ~ \times \\
&& ~~~\int_{[\lambda, 1]^{(m-1)}} \prod_{i=1}^{m-1} \lambda_i^{\frac{\beta}{2}(a+1)-1} ~(1-\lambda_i)^{\frac{\beta}{2}(b+1) - 1} (\lambda_i - \lambda)^{\beta} ~ \cdot ~\prod_{i<j} |\lambda_i - \lambda_j|^{\beta} ~d\lambda_1 \ldots d\lambda_{m-1}~.
\end{eqnarray*}

Note that the $m$ in the first numerator of the right hand side comes from choosing $\lambda_m$ as the smallest eigenvalue.

We can now make the variable change $x_i = \frac{1 - \lambda_i}{1-\lambda}$, which maps $\lambda_i$ from $[\lambda, 1]$ to $[0,1]$. The distribution  thus becomes
{\footnotesize 
\begin{eqnarray}
f_{min}(\lambda)  &=& \frac{m}{c_{\beta, a,b, m}} ~\lambda^{\frac{\beta}{2}(a+1)-1} ( 1- \lambda)^{\frac{\beta}{2}m (b+m) -1} ~\times \nonumber \\
&&  \int_{[0, 1]^{(m-1)}} \prod_{i=1}^{m-1}  x_i^{\frac{\beta}{2} (b+1)-1} ~ (1-x_i)^{\beta} ~ (1- x_i (1-\lambda))^{\frac{\beta}{2}(a+1) - 1}~ \prod_{i<j} |x_i - x_j|^{\beta} ~dx_1 \ldots dx_{m-1}~. \label{primo}
\end{eqnarray}}

We will now give a different expression for the integral on the right hand side by using equation (13.12) from \cite{forr_book}. For our particular case, this result specializes to the following Proposition:
\begin{proposition} \label{forr_2f1} We have
\begin{eqnarray*}
\frac{1}{c_{\beta, b, 1+2/\beta, m-1}} \int_{[0,1]^{(m-1)}} \prod_{i=1}^{m-1} x_i^{\frac{\beta}{2} (b+1)-1} ~ (1-x_i)^{\beta}  ~(1- x_i (1-\lambda))^{\frac{\beta}{2}(a+1) - 1}~ \prod_{i<j} |x_i - x_j|^{\beta}~dx_1 \ldots dx_{m-1} & = &  \\
&&  \hspace{-12.5cm} {}_2F_{1}^{2/\beta} ( 1 - \frac{\beta}{2}(a+1), \frac{\beta}{2}(b+m-1); \frac{\beta}{2}(b+2m-1)+1; (1-\lambda) I_{m-1})~,
\end{eqnarray*}
where $c_{\beta, b, 1 + 2/\beta, m-1}$ is given by \eqref{cee} and $I_{m-1}$ has the same meaning as in \eqref{jack_id}.
\end{proposition}

Thus, we obtain the following theorem:

\begin{theorem} \label{general_min}
The probability distribution of the smallest eigenvalue $\lambda$ of the $\beta$-Jacobi ensembles of parameters $a$ and $b$ and size $m$ is given by 
\begin{eqnarray}
f_{min} (\lambda) & =& C_{\beta, a, b, m} ~~\lambda^{\frac{\beta}{2}(a+1)-1}~(1- \lambda)^{\frac{\beta}{2}m(b+m)-1} ~\times  \nonumber \\
&& \hspace{1cm}  {}_2F_{1}^{2/\beta} ( 1 - \frac{\beta(a+1)}{2}, \frac{\beta(b+m-1)}{2}; \frac{\beta(b+2m-1)}{2}+1; (1-\lambda) I_{m-1})~, \label{expr}
\end{eqnarray}
with
\begin{eqnarray*}
C_{\beta, a,b, m} & = & \frac{m~ c_{\beta, b, 1+ 2/\beta, m-1}}{c_{\beta, a, b, m}} ~.
\end{eqnarray*}
\end{theorem}

Making the observation that $1- \lambda_{min}$ is the largest eigenvalue $\lambda_{max}$ of the Jacobi ensemble of parameters $b$, $a$, and size $m$, we immediately obtain the following corollary.

\begin{corollary} \label{general_max}
The probability distribution of the largest eigenvalue $\tilde{\lambda}$ of the $\beta$-Jacobi ensembles of parameters $a$ and $b$ and size $m$ is given by 
\begin{eqnarray}
f_{max} (\tilde{\lambda}) & =& \tilde{C}_{\beta, a, b, m} ~~(1-\tilde{\lambda})^{\frac{\beta}{2}(b+1)-1}~\tilde{\lambda}^{\frac{\beta}{2}m(a+m)-1} \times  \nonumber \\
&& \hspace{1cm}  {}_2F_{1}^{2/\beta} ( 1 - \frac{\beta(b+1)}{2}, \frac{\beta(a+m-1)}{2}; \frac{\beta(a+2m-1)}{2}+1; \tilde{\lambda} I_{m-1})~, \label{expr2}
\end{eqnarray}
with
\begin{eqnarray*}
\tilde{C}_{\beta, a, b, m} & = & \frac{m ~c_{\beta, a, 1+ 2/\beta, m-1}}{c_{\beta, a,b, m}} ~.
\end{eqnarray*}
\end{corollary}

A quick examination of the distribution \eqref{expr} suggests right away that, if $a$ is fixed or grows to $\infty$ much slower than $m$ or $b+m$, the smallest eigenvalue should go to $0$. In order to obtain asymptotics for the hypergeometric expression, we would then like to employ a transformation that would change $(1-\lambda)$ to  $\lambda$; however, for multivariate hypergeometrics, such formulae are only available when the series terminates (i.e.,  $1 - \beta(a+1)/2$ is a non-positive integer). In the absence of such a transformation for the general case (which is known in the classical case as a \emph{Kummer transformation}), we can also deal with the case when the multivariate hypergeometric can be reduced to a classical (single variable) one. This can be done if the the only partitions with corresponding non-zero coefficient in the series have only one part, i.e., $\beta(a+1)/2 - 1 = -\beta/2$. 

In the following sections, we will study the formula of Theorem \ref{general_min} for the particular values of $a,b$ and $\beta$ that place us in one of these two cases, and find asymptotics in various regimes. 

\subsection{Asymptotics for special cases}

Here we return to the motivating problem; what values of $\beta$, $a$, and $b$ are of interest in practice, and is it possible to use the hypergeometrics expressions given in Theorem \ref{general_min} to find the asymptotics?

For the \textbf{MP}, as given by Proposition \ref{jac_distr}, the distributions correspond to:
\begin{itemize}
\item[Case 1.] The real case: $\beta = 1$, $a = 0$, $b = n-2r$, and $m= r$.  What turns out to be crucial in this case is that $\beta/2 (a+1) -1 = - \beta/2$; we will therefore study the general case $ a = 2/\beta - 2$, $m=r$.
\item[Case 2.] The complex case: $\beta = 2$, $a = 0$, $b = n-2r$, and $m=r$. For this case, the crucial factor is that $\beta/2 (a+1) - 1 = 0$; this is a particular instance of $\beta/2(a+1) \in \mathbb{N_+}$, the positive integers (or natural numbers). We will analyze this latter, more general case. 
\end{itemize}

\begin{remark} It is worth noting that, although not connected directly with the \textbf{MP}, another case of interest relates to the quaternionic Haar measure. As per \cite{forr06}, if instead of an orthogonal or unitary matrix in the \textbf{MP} we start from a symplectic (quaternionic) one, the resulting distribution is the same as in \ref{jac_distr}, with $\beta = 4$, $a = 0$, $b=n-2r$, and $m=r$. This can be seen as a particular instance of Case 2; so Case 2 will cover both complex and quaternionic matrices. \end{remark}

In the following sections we study the asymptotics of the smallest eigenvalue distributions in each of these two cases.

\subsection{Case 1: $a = \frac{2}{\beta} - 2$} \label{case1}

In this case we have that $\beta/2(a+1) - 1 = - \beta/2$. We start by examining the smallest eigenvalue distribution as given by Theorem \ref{general_min}:

\begin{eqnarray}
f_{min} (\lambda) & =& C_{\beta, b, m} ~~\lambda^{-\frac{\beta}{2}}~(1- \lambda)^{\frac{\beta}{2}m(b+m)-1} \times  \nonumber \\
&& \hspace{1cm}  {}_2F_{1}^{2/\beta} (\frac{\beta}{2}, \frac{\beta(b+m-1)}{2}; \frac{\beta(b+2m-1)}{2}+1; (1-\lambda) I_{m-1})~;\label{eq_case1}
\end{eqnarray}
here
\begin{eqnarray*}
C_{\beta, b, m} & = & \frac{m~c_{\beta, b, 1+2/\beta, m-1}}{c_{\beta, 2/\beta - 2, b, m}} ~.
\end{eqnarray*}

After various cancellations, we obtain that
\begin{eqnarray*}
C_{\beta, b, m} & = & \frac{\frac{\beta}{2} m (b+m) ~ \Gamma \left (\frac{\beta}{2}(b+m-1)+1 \right ) \Gamma \left ( \frac{\beta}{2} (m-1)+1 \right )}{\Gamma \left ( 1 - \frac{\beta}{2} \right) \Gamma \left ( \frac{\beta}{2} (b+2m-1)+1 \right)}~.
\end{eqnarray*}

\begin{lemma} \label{hyp_red_case1}
We have

{\footnotesize \[
\!\!\!\!\!{}_2F_{1}^{\beta} \left (\frac{\beta}{2}, \frac{\beta(b+m-1)}{2}; \frac{\beta(b+2m-1)}{2} + 1; (1-\lambda)^{m-1} \right ) = {}_2F_{1} \left ( \frac{\beta(b+m-1)}{2}, \frac{\beta(m-1)}{2}; \frac{\beta(b+2m-1)}{2}+1 ; 1-\lambda \right)~,
\]}
where the second hypergeometric is no longer the matrix argument one, but the classical one (corresponding to matrix dimension $1$). 
\end{lemma}

\begin{proof}
The proof is based on two simple facts. One is that the series is absolutely convergent for $\lambda \in [0,1]$, and the second is that for any partition $\kappa$ of $k$,
\[
\left ( \frac{\beta}{2} \right )^{\beta}_{\kappa} = \delta_{\kappa, [k]} ~\left ( \frac{\beta}{2} \right)_k~,
\]
where the Pochhammer symbol in the right hand side is the classical falling factorial. As soon as $\kappa$ has more than one part, the contribution from that part will make the falling factorial $0$ by Definition \ref{pocch_jk}.

We can now see that
{\footnotesize
\begin{eqnarray*}
{}_2F_{1}^{\beta} \left (\frac{\beta}{2}, \frac{\beta(b+m-1)}{2}; \frac{\beta(b+2m-1)}{2} +1; (1-\lambda)^{m-1} \right ) & = & \sum_{k=0}^{\infty} \frac{ \left ( \frac{\beta}{2} \right )_{k} \left ( \frac{\beta(b+m-1)}{2} \right )_{k}}{k! ~\left ( \frac{\beta(b+2m-1)}{2} + 1 \right )_{k}} ~ C_{[k]}^{\beta} ((1-\lambda)I_{m-1})~,\\
&=&\sum_{k=0}^{\infty} \frac{ \left ( \frac{\beta}{2} \right )_{k} \left ( \frac{\beta(b+m-1)}{2} \right )_{k}}{k! ~\left ( \frac{\beta(b+2m-1)}{2} +1 \right )_{k}} ~ (1-\lambda)^{k}~ C_{[k]}^{\beta}(I_{m-1})~.
\end{eqnarray*}}
Using  \eqref{jack_id} for $\kappa = [k]$, along with the fact that 
in this case $j_{[k]}$ can be easily computed to be $$ j_{[k]}^{\beta} = (2/\beta)^k k! \prod_{i=1}^{k} \left ( 1+ \frac{2}{\beta}(i-1) \right) =  (2/\beta)^{2k} ~k!~ \left ( \frac{\beta}{2} \right )_{k}~, $$ 
after cancellation, Lemma \ref{hyp_red_case1} is proved. 
\end{proof}

To analyze the classical hypergeometric on the right hand side of \ref{hyp_red_case1}, we make use of the following transformation which can be found for example as formula 15.3.6 in \cite{Abr_Steg}:
\begin{eqnarray}
{}_2F_{1} ( a, b; c; z) & = & \frac{\Gamma(c) \Gamma(c-a-b)}{\Gamma(c-a) \Gamma(c-b)} {}_2F_{1}(a, b; a+b-c+1; 1-z) + \nonumber \\
&& ~~(1-z)^{c-a-b}~\frac{\Gamma(c) \Gamma(a+b-c)}{\Gamma(a) \Gamma(b)} {}_2F_{1}(c-a, c-b; c-a-b+1; 1-z)~. \label{form}
\end{eqnarray}
By plugging in $a = \frac{\beta(b+m-1)}{2}$, $b = \frac{\beta(m-1)}{2}$, and $c = \frac{\beta(b+2m-1)}{2} + 1$, as well as $z = 1-\lambda$, we obtain
{\footnotesize \begin{eqnarray}
\!\!\!\!\!\!\!\!\! {}_2F_{1} (\frac{\beta(b+m-1)}{2}, \frac{\beta(m-1)}{2}; \frac{\beta(b+2m-1)}{2}+1; 1-\lambda ) 
\!\!\!& = & \!\!\!
A_{\beta, b, m} ~ {}_2F_{1} (\frac{\beta(b+m-1)}{2}, \frac{\beta(m-1)}{2}; -\frac{\beta}{2}; \lambda) ~~+ \nonumber\\
& & \!\!\!\!\!\!\!\!\!\!\!\!\!\!\!\!\!\! \lambda^{1+ \frac{\beta}{2}} ~
B_{\beta, b, m} ~{}_2F_{1} (\frac{\beta m }{2}+1 , \frac{\beta(b+m)}{2} + 1; 2 + \frac{\beta}{2}; \lambda) \label{ugh}
\end{eqnarray}}
with
\begin{eqnarray*}
A_{\beta, b, m} & = & \frac{\Gamma\left(\frac{\beta(b+2m-1)}{2}+1\right)\Gamma \left(\frac{\beta}{2}+1\right)}{\Gamma \left(\frac{\beta m }{2}+1 \right) \Gamma \left(\frac{\beta(b+m)}{2} +1 \right)}~ ~~\mbox{and}\\
B_{\beta, b, m} & = & \frac{\Gamma\left(\frac{\beta(b+2m-1)}{2}+1\right)\Gamma \left(-1 - \frac{\beta}{2}\right)}{\Gamma \left(\frac{\beta(m-1)}{2} \right) \Gamma \left(\frac{\beta(b+m-1)}{2} \right)}~.
\end{eqnarray*}
The useful thing about the right-hand-side formula is that $m$ and $b+m$ act like essentially independent variables. If we combine \eqref{eq_case1}, the results of Lemma \ref{hyp_red_case1}, and \eqref{ugh}, followed by factoring out 
\[
F := \frac{\Gamma(\frac{\beta(b+2m-1)}{2} +1)\Gamma(1+ \frac{\beta}{2}) \Gamma( - \frac{\beta}{2}) }{\Gamma(\frac{\beta(b+m)}{2}+ 1)} 
\]
 from the right hand side of \eqref{ugh}, 
we can conclude the following. 

\begin{theorem} \label{gen_1}
The smallest eigenvalue of $\beta$-Jacobi ensembles with $ a= \frac{2}{\beta} - 2$ is given by
{\footnotesize \begin{eqnarray}
f_{\beta, b,m}(\lambda) ~~&= &~~\tilde{C}_{\beta, b,m} ~\lambda^{-\beta/2} ~ ( 1- \lambda)^{\beta m (b+m)/2-1} \nonumber \\
& & \hspace{-.5cm} \times ~ \left ( \frac{1}{\Gamma(-\frac{\beta}{2}) \Gamma( \frac{\beta m}{2} +1 )} {}_2F_{1} (\frac{\beta(b+m-1)}{2}, \frac{\beta(m-1)}{2}; -\frac{\beta}{2}; \lambda) \right . \nonumber\\
& & \hspace{-.25cm} \left . - ~\lambda^{1+\beta/2} \frac{1}{\Gamma ( \frac{\beta}{2} + 2) \Gamma(\frac{\beta(m-1)}{2})} ~\frac{\Gamma(\frac{\beta(b+m)}{2}+1)}{\Gamma(\frac{\beta(b+m-1)}{2})} ~{}_2F_{1} (\frac{\beta m}{2}+1, \frac{\beta(b+m)}{2}+1; 2+\frac{\beta}{2}; \lambda) \right) \label{distr_rn}
\end{eqnarray}}
where
\[
\tilde{C}_{\beta, b, m} =\frac{\Gamma \left ( -\frac{\beta}{2} \right) \Gamma \left( 1 + \frac{\beta}{2} \right)}{\Gamma \left ( 1- \frac{\beta}{2}\right)} ~  \frac{\beta}{2}m (b+m) ~\frac{\Gamma \left ( \frac{\beta}{2} (b+m-1)+1\right)}{\Gamma \left(\frac{\beta}{2}(b+m) + 1\right)} ~\Gamma \left ( \frac{\beta}{2}(m-1) +1 \right)~.
\]
\end{theorem}

We can now analyze the asymptotics in various regimes. 

\begin{itemize}
\item[\textbf{Regime 1:}]  $m$ fixed, $b \rightarrow \infty$. This is the case when the $\beta$-Jacobi ensembles ``turn'' into $\beta$-Laguerre, and so we are also computing here the exact distributions for the smallest eigenvalue distributions for $\beta$-Laguerre ensembles.

For this case, we use the following two facts:
\begin{eqnarray}
\frac{\Gamma(x + \alpha)}{x^{\alpha}~\Gamma(x)} &\rightarrow&  1~~\mbox{as}~~ x \rightarrow \infty ~~\mbox{and $\alpha$ is fixed} \label{f1} \\
{}_2F_{1}(a, b; c; \frac{x}{a}) & \rightarrow &{}_1F_{1}(b;c;z)~~\mbox{as}~~a\rightarrow \infty \label{f2}~.
\end{eqnarray}

If we make the transformation $y = \beta (b+m) \lambda/2$ and allow $b \rightarrow \infty$, the hypergeometric expression between the parentheses in \eqref{distr_rn} becomes
{\footnotesize \begin{eqnarray*} 
 \frac{{}_1F_{1} (\frac{\beta(m-1)}{2}; -\frac{\beta}{2}; 
y) }{\Gamma \left ( - \frac{\beta}{2} \right) \Gamma \left (\frac{\beta}{2} m + 1\right) } ~~- ~\left (
y \right )^{1+ \beta/2}~ \frac{{}_1F_{1} (\frac{\beta m}{2}+1; 2+ \frac{\beta}{2}; 
y )}{\Gamma \left (\frac{\beta}{2} + 2 \right) \Gamma \left (\frac{\beta(m-1)}{2} \right)} ~ = \frac{- \sin (\frac{\beta}{2} \pi)}{\pi} ~U \left ( \frac{\beta(m-1)}{2}; - \frac{\beta}{2}; 
y \right) ~,
\end{eqnarray*}}
where we used formula 13.1.3 from \cite{Abr_Steg}. 

We obtain thus the following result:

\begin{theorem} \label{a1}
As $b \rightarrow \infty$ while $m$ is kept fixed, the limiting distribution for the scaled smallest eigenvalue ($\beta(b+m) \lambda/2$) of the $\beta$-Jacobi ensemble with $a = 2/\beta - 2$ is given by  
\begin{eqnarray*}
f_m(y) & = & \frac{ 2}{\beta \pi} \Gamma \left ( 1 + \frac{\beta}{2} \right) \sin (\frac{\beta}{2} \pi ) ~ m \Gamma \left ( \frac{\beta(m-1)}{2}+1 \right)  ~y^{-\beta/2} e^{-my} ~U\left(\frac{\beta}{2}(m-1); - \frac{\beta}{2}; y \right)~.
\end{eqnarray*}
\end{theorem}

Note that when $\beta = 1$, this agrees with \cite{jiang09a} and \cite{edelman89a} (also taking into account that our scaling differs by a factor of $2$ from the ones used there.)

\item[\textbf{Regime 2:}] $m$, $(b+m)$ $\rightarrow \infty$. Note that this also covers the case $b$ fixed. We now use facts \eqref{f1} and \eqref{f2}, as well as
\begin{equation} \label{f4}
{}_2F_{1}(a, b; c; \frac{x}{ab}) \rightarrow  {}_0F_{1}(c; x)
\end{equation}
to obtain the limiting distribution of $y = \beta m (b+m) \lambda/2$.

We shall use the fact that the hypergeometric expression in \eqref{distr_rn}, when multiplied by $\Gamma( \beta m/2+1)$, in this regime, converges to 
\begin{equation*}
G_{\beta} (y) := \frac{1}{\Gamma\left(- \frac{\beta}{2} \right)} ~{} _0F_{1} \left (-\frac{\beta}{2}; \frac{\beta}{2} y \right ) - \left(\frac{\beta}{2} y \right)^{-\beta/2} \frac{1}{\Gamma \left ( 2 +\frac{\beta}{2}\right)}  ~{}_0F_{1}(2+\frac{\beta}{2}; \frac{\beta}{2} y) ~;
\end{equation*}
using the definition of the classical $_0F_1$ hypergeometric function as well as formulae 9.6.2 and 9.6.10 from \cite{Abr_Steg}, we obtain that 
\[
G_{\beta}(y) = - \frac{2\sin \left ( \frac{\beta}{2} \pi \right)}{\pi} ~ \left( \frac{\beta}{2}y \right)^{\frac{1}{2} + \frac{\beta}{4}} ~K_{1+ \frac{\beta}{2}}( \sqrt{2 \beta y})~,
\]
where $K_{\nu}$ is the well-known modified Bessel function.

Using this we can now conclude the following.
\begin{theorem} \label{a2}
As $m, (b+m) \rightarrow \infty$, the limiting distribution for the scaled smallest eigenvalue ($\beta m (b+m) \lambda/2$) of the $\beta$-Jacobi ensemble with $a = 2/\beta - 2$ is given by  
\begin{eqnarray*}
f(y) & = & \frac{4 \sin \left ( \frac{\beta}{2} \pi\right)}{\beta \pi} ~\left ( \frac{\beta}{2} \right)^{\frac{1}{2} - \frac{\beta}{4}}~ \Gamma \left ( 1+ \frac{\beta}{2} \right) ~ y^{1/2 - \beta/4} e^{-y}~ K_{1+ \frac{\beta}{2}}(\sqrt{2\beta y})~.
\end{eqnarray*}
\end{theorem}

Note that when $\beta = 1$, this once again agrees with the combined results of \cite{jiang09a} and \cite{edelman89a}. Our scaling differs from the one used there by a factor of $2$; in effect the distribution for this case is simply 
\[
f_m(y) = \frac{1+ \sqrt{2y}}{\sqrt{2y}} e^{-y - \sqrt{2y}}~.
\]

The importance of this result is that the limit is achieved regardless of how $(b+m)$ and $m$ grow to infinity. In the case of the \textbf{MP}, we have $m= r, ~b = n- 2r$, and the condition is then that both $r$ and $n-r$ grow to infinity. This is stronger than the results following \cite{jiang09a}, where the condition is that $r = o(n/\log n)$.

\end{itemize}

\subsection{Case 2: $(a+1) \beta / 2  \in \mathbb{N}_+$} \label{case2}

Note that this case covers the $\beta \in 2\mathbb{N}_{+}$ case analyzed in \cite{forr_book} (although it is true that Forrester computes not only the hard edge/soft edge limits for $\beta$ even, but also all limiting correlations). 

For this particular case we will use a different hypergeometric formula to express $f_m(\lambda)$, which will yield a terminating series, which can then be converted into an analyzable expression.

Let $(a+1) \beta/2 = k$, with $k$ a positive integer.

Instead of (13.12) from \cite{forr_book}, this time we use equation (13.7) from the same, which in this particular case specializes to the proposition below.

\begin{proposition} We have
{\footnotesize \begin{eqnarray*}
\frac{1}{c_{\beta, b, 1 + 2/\beta, m-1}} \int_{[0,1]^{(m-1)}} \prod_{i=1}^{m-1} x_i^{\beta(b+1)/2-1} ~ (1-x_i)^{\beta} ~(1- x_i (1-\lambda))^{k - 1} ~ \prod_{i<j} |x_i - x_j|^{\beta}~dx_1 \ldots dx_{m-1} & = &  \\
&&  \hspace{-8cm}{}_2F_{1}^{4/\beta} (1-m, -m-b+1 ; -2m - b + 1 - \frac{2}{\beta} ; \{1-\lambda\}^{k-1})~,
\end{eqnarray*}}
where $c_{\beta, 2k/\beta - 1,b ,m-1}$ is given by \eqref{cee}.
\end{proposition}

Note that the hypergeometric series is terminating; in fact, it is a polynomial of degree at most $(m-1)(k-1)$.

Also note that since $b >-1$, $m \geq 1$, and $\beta>0$, $2m+b-1+\frac{2}{\beta} \geq m-1$, the series is well-defined  (if,  for some $\kappa$,  $(-2m - b + 1 - \frac{2}{\beta})_{\kappa} = 0$, then necessarily $(1-m)_{\kappa} = 0$ as well, and so the potentially ``offending'' term in the series does not in fact exist). 

Under these conditions, using Proposition 13.1.7 from \cite{forr_book}, we can change the variables in the hypergeometric expression from $(1-\lambda)^{k-1}$ to $\lambda^{k-1}$:
{\footnotesize \begin{eqnarray*}
_2F_1^{4/\beta} (1-m, -m-b+1; -2m-b+1- \frac{2}{\beta}; \{1-\lambda\}^{k-1}) & = & \frac{_2F_1^{4/\beta} (1-m, -m-b+1; 2 + \frac{2}{\beta}(k-1); \{\lambda\}^{k-1})}{_2F_1^{4/\beta} (1-m, -m-b+1; 2 + \frac{2}{\beta}(k-1); \{1\}^{k-1})}~.
\end{eqnarray*}}

According to equation (13.14) of the same \cite{forr_book}, 
\[
_2F_1^{4/\beta} (1-m, -m-b+1; 2 + \frac{2}{\beta}(k-1); \{1\}^{k-1}) = \prod_{j=1}^{k-1} \frac{\Gamma \left ( 2 + \frac{2}{\beta} j \right) \Gamma \left ( 2m + b + \frac{2}{\beta} j \right) }{\Gamma \left ( m + 1 + \frac{2}{\beta} j \right) \Gamma \left (m+b+1+ \frac{2}{\beta} j \right)}~:=A_{m, b, \beta, k}.
\]

We now group the left and, respectively, right terms in the fraction above, and rewrite the product as follows:
\begin{eqnarray*}
A_{m, b, \beta, k} & = & \prod_{j=1}^{k-1} \prod_{i=1}^{m-1} \frac{1}{i+1+\frac{2}{\beta} j} ~(m+b+i + \frac{2}{\beta}{j}) \\
& = & \prod_{j=1}^{k-1} \prod_{i=1}^{m-1} \frac{1}{(i+1)\frac{\beta}{2} + j} ~((m+b+i) \frac{\beta}{2} + j)~\\
& = & \prod_{i=1}^{m-1} \frac{ \Gamma \left ( 1 + (i+1) \frac{\beta}{2} \right)}{ \Gamma \left (k+(i+1) \frac{\beta}{2}\right)} ~ \frac{\Gamma \left (k + \frac{\beta}{2} (m+b+i) \right) }{\Gamma \left (1+ \frac{\beta}{2} (m+b+i) \right)}~.
\end{eqnarray*}

Putting everything together, we obtain
\begin{theorem} \label{gen_2}
For the case when $\beta/2(a+1)=k \in \mathbb{N}_+$, the smallest eigenvalue of the $\beta$-Jacobi ensemble is given by 
\begin{eqnarray}
f_m(\lambda) & = & W_{m,b, \beta, k} ~\lambda^{k-1} (1- \lambda)^{\frac{\beta}{2} m (b+m) - 1} ~ \nonumber \\
& & ~\times ~_2F_1^{4/\beta} (1-m, -m-b+1; 2 + \frac{2}{\beta}(k-1) ; \{\lambda\}^{k-1}) ~,
\label{due} \end{eqnarray}
with
\begin{eqnarray*} 
W_{m,b, \beta, k} & = & \frac{m ~c_{\beta, b, 1+ 2/\beta,  m-1}}{c_{\beta, 2k/\beta-1, b,m} A_{m,b,\beta,k}}~.
\end{eqnarray*}
After appropriate cancellations, we obtain that 
\begin{eqnarray} \label{const}
W_{m,b, \beta, k} & = & \frac{\Gamma \left ( 1+ \frac{\beta}{2} \right) }{\Gamma(k) \Gamma \left ( k + \frac{\beta}{2} \right) } ~\cdot~ \frac{m \Gamma \left ( k + \frac{\beta}{2}m \right)}{\Gamma \left (1 + \frac{\beta}{2} m \right)}~\cdot~ \frac{\Gamma \left ( k + \frac{\beta}{2}(b+m)\right)}{\Gamma \left ( \frac{\beta}{2}(b+m) \right)}~.
\end{eqnarray} 
\end{theorem}
As in the previous case, we can now start to analyze the asymptotics in two regimes.

\begin{itemize}
\item[\textbf{Regime 1:}]  $m$ fixed, $b \rightarrow \infty$. Using \eqref{fa1} and \eqref{f1}, we can obtain the distribution of the scaled eigenvalue $y = (b+m) \lambda$, as follows. 

\begin{theorem} \label{b1}
As $b \rightarrow \infty$ while $m$ is kept fixed, the limiting distribution for the scaled smallest eigenvalue ($(b+m) \lambda$) of the $\beta$-Jacobi ensemble with $\beta(a+1)/2 = k \in \mathbb{N}_+$ is given by  
\begin{eqnarray*}
f_m(y) & = & \frac{\left ( \frac{\beta}{2} \right)^{k} \Gamma \left ( 1+ \frac{\beta}{2} \right) }{\Gamma(k) \Gamma \left ( k + \frac{\beta}{2} \right) }  ~\cdot~ \frac{m \Gamma \left ( k + \frac{\beta}{2}m \right)}{\Gamma \left (1 + \frac{\beta}{2} m \right)} ~y^{k-1} ~e^{-\beta my/2} ~ _1F_1^{4/\beta}(1-m; 2+ \frac{2}{\beta}(k-1); \{-y\}^{k-1})~.
\end{eqnarray*}
\end{theorem}

Note that in the case $\beta = 2$, $k = 1$, which is the case 2 of the \textbf{MP}, the hypergeometric above simply becomes a constant which cancels the more complicated terms of the normalization, and yields that 
\[
f_m(y) = m e^{-my/2}~, 
\]
which is consistent with the results of \cite{jiang09a} and \cite{edelman89a}. 

\item[\textbf{Regime 2:}] $m, ~(b+m) \rightarrow \infty$. This covers the case when $b$ is actually fixed. Again, using \eqref{fa1}, \eqref{fa2}, and \eqref{f1}, the distribution of the scaled eigenvalue $y = m (b+m) \lambda$ is given below.

\begin{theorem} \label{b2}
As $m, (b+m) \rightarrow \infty$, the limiting distribution for the scaled smallest eigenvalue ($m (b+m) \lambda$) of the $\beta$-Jacobi ensemble with $\beta(a+1)/2 = k \in \mathbb{N}_+$ is given by  
\begin{eqnarray*}
f(y) & = & \frac{\left ( \frac{\beta}{2} \right)^{2k-1} \Gamma \left ( 1+ \frac{\beta}{2} \right) }{\Gamma(k) \Gamma \left ( k + \frac{\beta}{2} \right) }~\cdot~ y^{k-1}~e^{-\beta y/2}~  _0F_1^{4/\beta}(2+ \frac{2}{\beta}(k-1); \{y\}^{k-1})~.
\end{eqnarray*}
\end{theorem}

Once again, when $\beta = 2$ and $k= 1$, the hypergeometric expression in the above is constant, and the answer is
\[f(y) = e^{-y}~.\]
Again, this is consistent with \cite{jiang09a} and \cite{edelman89a}. The advantage over \cite{jiang09a} is that when $b = n-2r$ and $m=r$, as is the case in the \textbf{MP}, there are no conditions over the growth to $\infty$ of $r$ and $n-r$; in \cite{jiang09a}, $r$ must be $o(n/\log n)$.

\end{itemize}

\section{Numerical Experiments} \label{numex}

We have tested numerically all the formulae obtained in the previous sections. 
All plots were obtained in MATLAB; to compute the hypergeometric functions involved we have used Koev's easy-to-use hypergeometric series package \cite{plamen_hyp, plamen_hyp2}. For the Monte Carlo tests, we have used the convenient to work with bidiagonal models for $\beta$-Jacobi distributions, given below:
\begin{eqnarray*}
J_{\beta, n, a,b} & = & B_{\beta, n, a, b} \cdot  B_{\beta, n, a, b}^{T}~, ~~~\mbox{where} \\
B_{\beta,n, a, b} & \sim & \left [ \begin{array}{ccccc} c_n & - s_{n} c'_{n-1} & & & \\
& c_{n-1} s'_{n-1} & -s_{n-1} c'_{n-2} & & \\
& & c_{n-2} s'_{n-2} & \ddots & \\
& & & \ddots & - s_2 c'_1 \\
& & & & c_1 s'_1 \end{array} \right ]~,
\end{eqnarray*} 
where all $c_i$'s and $c'_i$'s are independent variables distributed as follows:
\[
c_i \sim \sqrt{Beta \left ( \frac{\beta}{2}(a+i),~ \frac{\beta}{2}(b+i)\right)}~~, ~~~~c'_i \sim \sqrt{Beta \left ( \frac{\beta}{2}i,~\frac{\beta}{2}(a+b+1+i) \right)}~,
\]
and $s_i = \sqrt{1-c_i^2}$, for all $1 \leq i \leq n$, while $s'_i= \sqrt{1-{c'_i}^2}$, for al $1 \leq i \leq (n-1)$. Here $Beta(s,t)$ stands for the well-known distribution with pdf proportional to $x^{s-1}(1-x)^{t-1}$ on $[0,1]$.

This matrix model, together with a proof that the matrix $J_{\beta, n, a,b}$ has eigenvalue pdf given by the Jacobi ensemble with given parameters, is a beautiful result of Sutton's \cite{sutton05}. 

Figure \ref{fig_gen} is an illustration of Theorem \ref{general_min}. The solid red line corresponds to the exact distribution of the smallest eigenvalue, as given by the theorem, for the case when $\beta = 1.75$, $m=4$, $a = 2.3$, and $b = 2.5$.  The normalized histogram (obtained using the {\tt histnorm}) represents the results of $10,000$ Monte Carlo tests using the bidiagonal matrix models. 

\vspace{1cm}

\begin{figure}[!ht]
\begin{center}

\vspace{-4cm}

\includegraphics[height=4.5in]{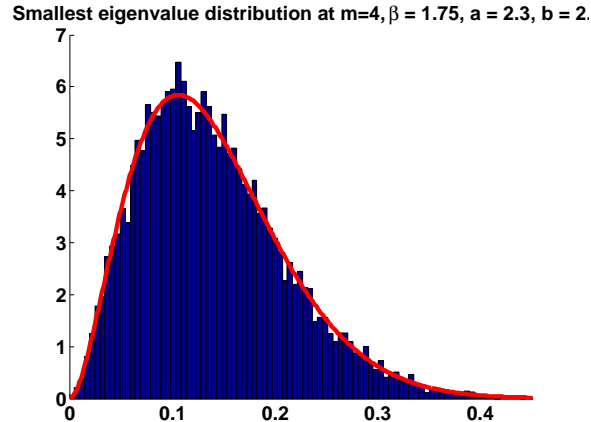}

\vspace{-3.5cm}

\caption{{\small The solid red line represents the theoretical distribution, while the normalized histogram represents the results of a Monte Carlo experiment with $10,000$ trials, using $J_{\beta, n, a, b}$.}} \label{fig_gen}
\end{center}

\end{figure}

\vspace{-.5cm}

Figures \ref{f_a1} and \ref{f_a2} represent a comparison between the theoretical results of \textbf{Case 1:} $\mathbf{a = \frac{2}{\beta} - 2}$, respectively, Theorems \ref{a1} and \ref{a2} (illustrated by the solid red lines), and Monte Carlo experiments with $10,000$, respectively $5,000$ trials (represented by the normalized histograms). Note that convergence occurs fairly quickly; for Figure \ref{f_a1}, which considers the case $b \rightarrow \infty$, the histogram of $\beta (b+m)/2 \lambda_{min}$ is very close to the plot of the asymptotical distribution, even though $b = 10$. For Figure \ref{f_a2}, we only need to take $b=5$ and $m=5$ to see how close the histogram of $\beta m (b+m)/2 \lambda_{min}$ is to the solid line representing the asymptotical distribution. For this latter case, also note the singularity at $0$; this is caused by the singularity of the Bessel function in the asymptotical formula, as given by Theorem \ref{a2}.


\begin{figure}[!ht]
\begin{center}

\vspace{-3cm}

\includegraphics[height=4.5in]{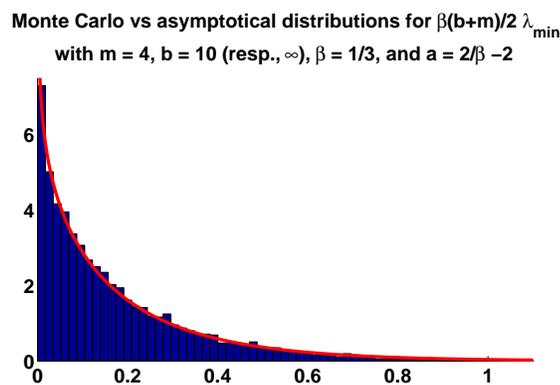}

\vspace{-3.5cm}

\caption{{\small The solid red line represents the asymptotical $(b = \infty)$ distribution, while the normalized histogram represents the results of a Monte Carlo experiment for $b = 10$, with $10,000$ trials.}} \label{f_a1}
\end{center}

\end{figure}

\begin{figure}[!ht]
\begin{center}

\vspace{-3cm}

\includegraphics[height=4.5in]{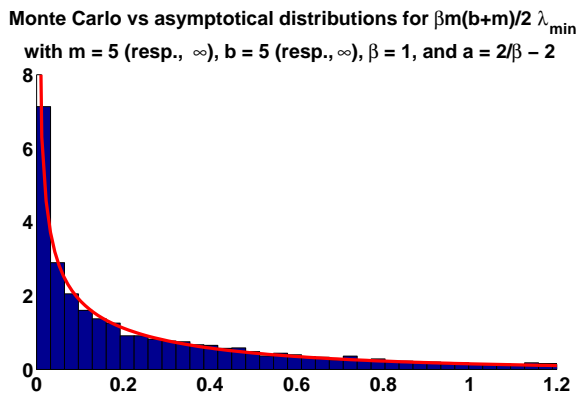}

\vspace{-3.5cm}

\caption{{\small The solid red line represents the asymptotical $(m, b = \infty)$ distribution, while the normalized histogram represents the results of a Monte Carlo experiment for $m=5$, $b = 5$, with $5,000$ trials.}} \label{f_a2}
\end{center}

\end{figure}

\pagebreak
Finally, figures \ref{f_b1} and \ref{f_b2} represent comparisons between the theoretical results of \textbf{Case 2:} $\mathbf{a = \frac{2k}{\beta} - 1}$, specifically, Theorems \ref{b1} and \ref{b2}. The exact distributions given by the theorems are represented by the solid red lines, while the histograms are the results of Monte Carlo tests with $10,000$ trials. This time, one has to increase $\beta$ to $50$ to obtain significant results; however, since the distributions represent asymptotics, that is still rather small. Figure \ref{f_b1} covers the case $\beta \rightarrow \infty$, that is, Theorem \ref{b1}, and the histogram is of the distribution on $(b+m) \lambda_{min}$, while Figure \ref{f_b2} covers the case $m, \beta \rightarrow \infty$ (Theorem \ref{b2}), and the histogram reflects the empirical distribution of $m(b+m) \lambda_{min}$.

\begin{figure}[!ht]
\begin{center}

\vspace{-3cm}

\includegraphics[height=4.5in]{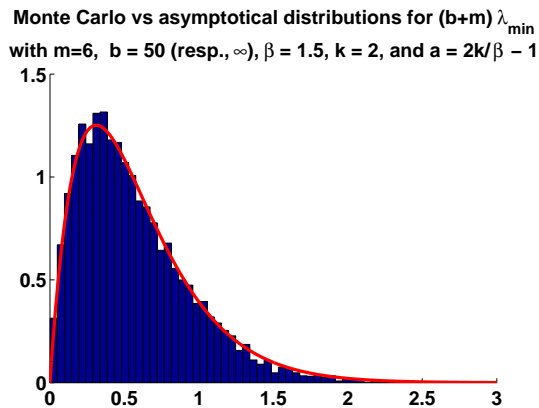}

\vspace{-3cm}

\caption{The solid red line represents the asymptotical $(b = \infty)$ distribution, while the normalized histogram represents the results of a Monte Carlo experiment for $b = 50$, with $10,000$ trials.} \label{f_b1}
\end{center}

\end{figure}

\begin{figure}[!ht]
\begin{center}

\vspace{-3cm}

\includegraphics[height=4.5in]{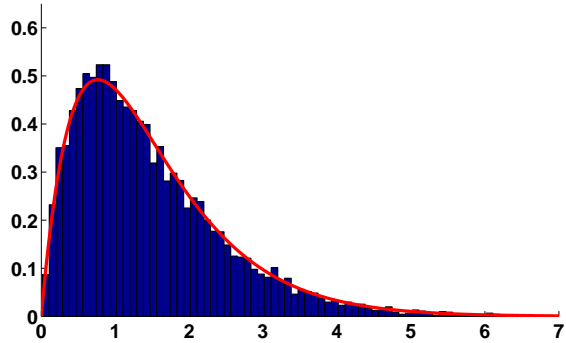}

\vspace{-3cm}

\caption{The solid red line represents the asymptotical $(m, b = \infty)$ distribution, while the normalized histogram represents the results of a Monte Carlo experiment for $m=15$, $b = 50$, with $10,000$ trials.} \label{f_b2}
\end{center}

\end{figure}

\section*{Acknowledgements} Ioana would like to thank James Demmel, Peter Forrester, Tiefeng Jiang, and Eric Rains for useful discussions. Ioana's work is supported by NSF CAREER Award DMS-0847661. 

\bibliography{bib_09_10}
\bibliographystyle{plain}

\end{document}